\documentclass[reqno,10pt]{amsart}
\usepackage{amssymb}
\usepackage{latexsym}
\usepackage{hyperref}
\usepackage{amsmath}
\usepackage{amscd}
\usepackage{color}
\usepackage{enumerate}
\usepackage{amsfonts}
\usepackage{graphicx}
\usepackage{mathrsfs}
\usepackage{pifont}
\usepackage{setspace}

\usepackage{tabu}
\newcommand{\eps}{{\varepsilon}}
\theoremstyle{plain}
\newtheorem{theorem}{Theorem}
\newtheorem{proposition}[theorem]{Proposition}
\newtheorem{lemma}[theorem]{Lemma}
\newtheorem{corollary}[theorem]{Corollary}

\theoremstyle{definition}
\newtheorem{definition}[theorem]{Definition}
\newtheorem{remark}[theorem]{Remark}

\numberwithin{equation}{section}
\numberwithin{theorem}{section}
\let\Re=\undefined\DeclareMathOperator*{\Re}{Re}
\let\Im=\undefined\DeclareMathOperator*{\Im}{Im}

\def\ge{\geqslant}
\def\le{\leqslant}
\def\geq{\geqslant}
\def\leq{\leqslant}
\def\b{\boldsymbol}

\def\R{\mathbb{R}}
\def\C{\mathbb{C}}

\begin{document}
\title[Scattering for NLS system in $\mathbb{R}^5$]{Scattering For  three waves Nonlinear Schr\"odinger System with mass-resonance in 5D}

\author[Fanf. Meng]{Fanfei Meng}
\address{The Graduate School of China Academy of Engineering Physics,
	Beijing, 100088, P.R. China }
\email{mengfanfei17@gscaep.ac.cn}
\author[S. Wang]{Sheng Wang}
\address{Shanghai Center for Mathematical Sciences, Fudan University, Shanghai 200433, P.R. China}
\email{19110840011@fudan.edu.cn}	
\author[Chengb. Xu]{Chengbin Xu}
\address{The Graduate School of China Academy of Engineering Physics,
	Beijing, 100088, P.R. China }
\email{xuchengbin19@gscaep.ac.cn}

\maketitle

\begin{abstract}
In this paper,  we study the dynamics behavior of the \eqref{NLS system} with three waves interaction in the energy space $H^1(\mathbb{R}^5) \times H^1(\mathbb{R}^5)\times H^1(\mathbb{R}^5) $. Inspired by B. Dodson and J. Murphy in \cite{Dodson2018},
we establish an  interaction Morawetz estimate for the \eqref{NLS system}, together with the criterion which proved by Tao-Dodson--Murphy we can get the scattering under the ground state in energy space with mass-resonance. Under the radial assumption, we can remove the mass-resonance condition.
\\

\noindent\textbf{Key words:} nonlinear Schr\"odinger system, scattering theory, interaction Morawetz estimate.
\end{abstract}

\maketitle
\section{Introduction}
Considering the Cauchy problem for the quadratic nonlinear Schr\"odinger system:
\begin{equation}\tag{$\text{NLS system}$}
\left\{
\begin{aligned}
& i \partial_{t}\b{\rm u} + A \b{\rm u}  =  \b{\rm f}(\b{\rm u}), \quad (t, x) \in \mathbb{R} \times \mathbb{R}^d,\\
& \b{\rm u}(0,x)= \b{\rm u}_0(x),
\end{aligned}
\right.
\label{NLS system}
\end{equation}
where $ \b{\rm u} $, $ \b{\rm u}_0 $, and $ \b{\rm 0} $ are all vector-valued complex functions with three components  defined as follow:
\begin{equation}
\b{\rm u} := \left(
\begin{aligned}
u^1 \\
u^2 \\
u^3
\end{aligned}
\right),~\b{\rm u_0} := \left(
\begin{aligned}
u^1_0 \\
u^2_0 \\
u^3_0
\end{aligned}
\right),
\end{equation}
and  $ A $ is a $ 3 \times 3$ matrix, $ \b{\rm f} : \C^3 \to \C^3  $ as follow
\begin{equation}\label{mass-res}
A=
\begin{pmatrix}
\kappa_1 \Delta & 0 & 0  \\
0                & \kappa_2 \Delta  & 0   \\
0                & 0                      & \kappa_3 \Delta
\end{pmatrix}, ~ \b{\rm f}(\b{\rm u}):= \left(
\begin{aligned}
-\overline{u^2}u^3 \\
-\overline{u^1}u^3 \\
-u^1 u^2
\end{aligned}
\right),
\end{equation}	
where $u^{i}$ : $\mathbb{R} \times \mathbb{R}^{d} \rightarrow \mathbb{C} $ are unknown functions, $\kappa_{i} \in \left(0,\infty\right)$ are real number $\left(i=1,2,3\right)$.

In terms of physics, the quadratic nonlinearity of \eqref{NLS system} arises from a model in the nonlinear optics describing the second harmonic generation and Raman amplification phenomenon  in plasma, this process is a nonlinear instability phenomenon (see \cite{Colin2009} for more detail).
In fact, \eqref{NLS system} is the form of infinite dimensional Hamilton if we treat $ \C \sim \R^2 $. 

\begin{equation}\tag{$\text{IDH}$}
\partial_{t} \b{\rm u} = S_6 \frac{\partial H}{\partial \overline{\b{\rm u}}}, 
\label{IDH}
\end{equation}
where $ S_6 $ is the standard $ 6 \times 6 $ symplectic matrix and $ H = H(\b{\rm u}) $ is the Hamiltion with 

\[
H(\b{\rm u}) = \sum_{i=1}^3 \frac{\kappa_i}{2}\Vert u^i \Vert_{{\dot H}^1}^2 - \Re \int_{\mathbb{R}^d} \overline{u^1}\overline{u^2}u^3 {\rm d}x. 
\]

Thus, by Emmy Noether's theorem, the solutions to \eqref{NLS system} conserve the \emph{mass}, \emph{energy} and \emph{momentum}, corresponding to the transfomations of rotation of phase, translation of time and translation of space respectively, defined as follows
\begin{equation*}
\begin{aligned}
M(\b{\rm u}) :& = \frac{1}{2}\Vert u^1 \Vert_{L^2}^2 + \frac{1}{2} \Vert u^2 \Vert_{L^2}^2 +\Vert u^3 \Vert_{L^2}^2  \equiv M(\b{\rm u}_{0}), \\
E(\b{\rm u}) :& = K(\b{\rm u}) -V(\b{\rm u}) \equiv E(\b{\rm u}_{0}), \\
P(\b{\rm u}) :& = \Im \int_{\mathbb{R}^d} \left( \overline{u^1}\nabla u^1 + \overline{u^2}\nabla u^2 + \overline{u^3}\nabla u^3 \right) {\rm d}x \equiv P(\b{\rm u}_0),
\end{aligned}
\end{equation*}
where
\[
\begin{aligned}
& (\text{kinetic energy}) \quad & K(\b{\rm u}): & = \frac{\kappa_1}{2}\Vert u^1 \Vert_{{\dot H}^1}^2 + \frac{\kappa_2}{2} \Vert u^2 \Vert_{{\dot H}^1}^2 + \frac{\kappa_3}{2} \Vert u^3 \Vert_{{\dot H}^1}^2  , \\
& (\text{potential energy}) \quad & V(\b{\rm u}) :& = \Re \int_{\mathbb{R}^d} \overline{u^1}\overline{u^2}u^3 {\rm d}x.
\end{aligned}
\]

The equation \eqref{NLS system} also shares the scaling invariance 
\[
\b{\rm u}_{\lambda}(t, x) = \lambda^2 \b{\rm u} \left( \lambda^2 t, \lambda x \right)
\]
for $ \lambda > 0 $.
If we set ${{\dot{\rm H}}^s}:= {\dot H}^s \times {\dot H}^s \times {\dot H}^s $,  then the critical regularity of Sobolev space is $ {\rm \dot{H}}^{s_c} $,
i.e. $ \Vert \b{\rm u}_{\lambda} \Vert_{{\rm \dot{H}}_{x}^{s_c}(\mathbb{R}^d)} =  \Vert \b{\rm u} \Vert_{{\rm \dot{H}}_{x}^{s_c}(\mathbb{R}^d)} $ where $ s_c = \frac{d}{2} - 2 $.
Therefore, the equation \eqref{NLS system} is called mass-subcritical if $ d \le 3 $, mass-critical if $ d=4 $, energy-subcritical if $ d=5 $, and energy-critical if $ d=6 $.


Now, we review some basic facts about the nonlinear Schr\"odinger equation:
\[\tag{$\text{NLS}$}
i\partial_{t}u + \Delta u = \mu \vert u \vert^{p-1} u , \quad (t, x) \in \mathbb{R} \times \mathbb{R}^d,
\label{NLS}
\]
The case $\mu=1$ is the defocusing case, while $\mu=-1$ gives the focusing case.
For the defoucsing case, the potential energy has the same sign as the kinetic energy; while for the focusing case,  the contray is the case. In contrast to the classical nonlinear  Schr\"odinger equation \eqref{NLS}, the \eqref{NLS system} do not have classification of the so-called focusing and defocusing, since we can not figure out the sign  of potential energy.

Both defocusing case ($ \lambda > 0 $) and focusing case ($ \lambda < 0 $) of \eqref{NLS} have been studied in a large amount of literature,
such as \cite{Bourgain1999, Colliander2008, Dodson2012,  Dodson2017, Duyckaerts2007, Killip2010, Miao2013, Miao2014, Miao2017, Visan2007}
by J.Bourgain, T. Tao, B. Dodson, T. Duyckaerts, R. Killip, C. Miao, M. Visan and so on.

Similarly to the classical \eqref{NLS}, we can define the Galilean transformation to the \eqref{NLS system} as follows:
\[
\left(
\begin{aligned}
u^1(t, x) \\
u^2(t, x) \\
u^3(t,x)
\end{aligned}
\right) \to \left(
\begin{aligned}
e^{i\frac{x\cdot\xi}{\kappa_1}}e^{-it\frac{\vert\xi\vert^2}{\kappa_1}}u^1(t, x-2t\xi) \\
e^{i\frac{x\cdot\xi}{\kappa_2}}e^{-it\frac{\vert\xi\vert^2}{\kappa_2}}u^2(t, x-2t\xi) \\
e^{i\frac{x\cdot\xi}{\kappa_3}}e^{-it\frac{\vert\xi\vert^2}{\kappa_3}}u^3(t, x-2t\xi)
\end{aligned}
\right).
\]
It is easy to check that the \eqref{NLS system} is invariant under Galiean transformation only if the coefficient satisfies
\begin{equation}\label{mr}
\frac{1}{\kappa_3}=\frac{1}{\kappa_1} + \frac{1}{\kappa_2}.
\end{equation}
So we call the \eqref{mr} is the mass-resonance condition.

In this article, we consider energy-subcritical case (d=5) of \eqref{NLS system} under the mass-resonance condition \eqref{mr}.
By solution, we mean a function $ \b{\rm u} \in C_{t}(I, {\rm H}_{x}^1(\mathbb{R}^5)) $ on an interval $ I \ni 0 $ satisfying the Duhamel formula
\[
\b{\rm u}(t) = \mathcal{S}(t) \b{\rm u}_0 + i \int_{0}^{t} \mathcal{S}(t-s) \b{\rm f}(\b{\rm u}(s)) {\rm d}s
\]
for $ t \in I $,
where we denote Schr\"odinger flow $ \mathcal{S}(t)\b{\rm w} =: \left( e^{\kappa_1 it\Delta}w^1, e^{\kappa_2 it\Delta}w^2, e^{\kappa_3 it\Delta}w^3 \right)^T $ for any $ \b{\rm w} =: (w^1, w^2, w^3)^T $.
In order to study the well-posedness of \eqref{NLS system}, we need Strichartz estimates.
By using the standard contraction mapping theorem, we can show: $ \exists ~ \delta_0 > 0 $ such that, if
\begin{equation*}
\left\Vert \mathcal{S}(t) \b{\rm u}_0 \right\Vert_{{\rm L}_{t}^6(I, {\rm L}_{x}^3(\mathbb{R}^5))} < \delta_0,
\end{equation*}
then there exists a unique global solution $ \b{\rm u}(t) = (u^1(t), u^2(t), u^3(t))^T $ to \eqref{NLS system}.
For the large initial data, we have solution $ \b{\rm u}(t) $ with a maximal interval of existence $ I_{\max} = (T_-(\b{\rm u}), T_+(\b{\rm u})) $.
A global solution $ \b{\rm u} $ ``scatters'', i.e.
there exists $ \b{\rm u}_{\pm} \in {\rm H}^1(\mathbb{R}^5) $ such that
\begin{equation*}
\lim_{t \to \pm \infty} \left\Vert \b{\rm u}(t) - \mathcal{S}(t) \b{\rm u}_{\pm} \right\Vert_{{\rm H}^1(\mathbb{R}^5)} = 0.
\end{equation*}

\begin{definition}[Trivial scattering solution, \cite{Masaki2021}]
	We say a solution to a nonlinear equation is a \emph{trivial scattering solution} if it also solves the corresponding linear equation.
	(In other words, the solution which does not have nonlinear interaction.)
\end{definition}
\begin{remark}
	For \eqref{NLS}, zero is the only trivial scattering solution.
	While for \eqref{NLS system}, $ (u^1, u^2, u^3)^T = (0, 0, e^{\kappa_3 it\Delta} u^3_0) $ is  a trivial scattering solution for any $ u^3_0 \in  H^1(\R^5) $.
\end{remark}

The \eqref{NLS system} admits a global but nonscattering solution
\[
(u^1(t, x), u^2(t, x), u^3(t,x))^T = (e^{it}\phi_1(x), e^{it}\phi_2(x), e^{2it}\phi_3(x))^T,
\]
where $\b{\rm Q} := (\phi_1, \phi_2,\phi_3)^T \in \Xi$ is a non-negative radial solution to the elliptic system
\begin{equation}
\left\{
\begin{aligned}
& \phi_1- \kappa_1\Delta \phi_1 = \phi_2 \phi_3, \\
& \phi_2- \kappa_2\Delta \phi_2 = \phi_1 \phi_3, \\
& 2 \phi_3- \kappa_3\Delta \phi_3 =\phi_1 \phi_2
\end{aligned}
\right.
\quad x \in \mathbb{R}^5,
\label{gs}
\end{equation}
where $ \Xi :=\left\{ \b{\rm G} \in {\rm H}^1(\R^5)  ~ \big\vert ~ V(\b{\rm G}) \neq 0 \right\} $.

We call $\b{\rm Q} := (\phi_1, \phi_2,\phi_3)^T$ the ``{\bf ground state}''.
In \cite{Hamano2018}, M. Hamano determined the global behavior of the solutions to the system with data below the ground state and proved a blowing-up result if the data had finite variance or was radial.

Our main result in this paper is follows:
\begin{theorem}\label{main}
For the \eqref{NLS system} under the mass-resonance condition \eqref{mr}.
If the initial data $ \b{\rm u}_0 \in {\rm H}^1(\mathbb{R}^5) $ satisfies
$ M( \b{\rm u}_0 ) E( \b{\rm u}_0 ) < M( \b{\rm Q} ) E( \b{\rm Q} ) $ and $ M(\b{\rm u}_0) K(\b{\rm u}_0) \le M(\b{\rm Q}) K(\b{\rm Q}) $,
then the solution of \eqref{NLS system} is global and scatters in $ {\rm H}^1(\R^5) $.
\end{theorem}

M. Hamano in \cite{Hamano2018} originally studied the \eqref{NLS system} with two waves interaction under mass-resonance, through the use of concentration compactness by Kenig-Merle in \cite{Kenig2006}. Later, using concentration compactness again, M. Hamano, T. Inui and K. Nishimura in \cite{Hamano2019} studied the scattering for \eqref{NLS system} with two waves interaction  without mass-resonance in the radial case. Inspired by B. Dodson and J. Murphy in \cite{Dodson2018}, we have present a proof for the non-radial \eqref{NLS system} with two waves interaction under the mass-resonance in \cite{Meng2020} and we continue the work of scattering result corresponding to \eqref{NLS system} here.

\begin{remark}
In fact, the results of Theorem \ref{main} essentially holds for any $ \kappa_i> 0,  (i=1, 2, 3) $ if the initial data $ \b{\rm u}_0 $ is radial. Without mass-resonance \eqref{mr}, we have to add the radial assumption owing to the lack of Galilean invariance when we study the scattering for \eqref{NLS system}.
Under this circumstance, if the \eqref{NLS system} produce the mass concentration, the assumption of radial implies it must be near the origin.
Thus, we only need to use the simpler Morawetz estimate instead of Proposition \ref{decay} to verify the condition of scattering criterion.
\end{remark}

Our proof of Theorem \ref{main} consists of two steps:
firstly, we establish a scattering criterion as follows, using the method from B. Dodson, J. Murphy \cite{Dodson2016} and T. Tao \cite{Tao2004}.
Secondly, in order to verify the condition of the above criterion,
we prove a certain decay estimate, which can be deduced from an interaction Morawetz estimate.
The proof of the following interaction Morawetz estimate relies on the Galilean invariance of \eqref{NLS system}.
\begin{proposition}[Interaction Morawetz estimate]\label{decay}
Let $ \b{\rm u}_0, \b{\rm Q}, I $ be as in Theorem \ref{scat}, and suppose further that
\begin{equation}
M(\b{\rm u}_0) = E(\b{\rm u}_0) = E_0.
\label{E0}
\end{equation}
Let $ \b{\rm u} : \mathbb{R} \times \mathbb{R}^5 \to \mathbb{C} $ be the corresponding global solution to \eqref{NLS system}.
Then there exists $ \delta > 0 $ such that for $ R_0 = R_0(\delta, M(\b{\rm u}), \b{\rm Q}) $ sufficiently large,
\begin{equation*}
\begin{aligned}
&\frac{\delta}{JT_0} \int_{I} \int_{R_0}^{R_0e^{J}}\frac{1}{R^5} \iiint L^{\xi}_{\kappa_1\kappa_2\kappa_3} \chi^{2} \left( \frac{x - s}{R} \right) \Gamma^{2} \left( \frac{y - s}{R} \right) N_{\kappa_1\kappa_2 \kappa_3} {\rm d}x {\rm d}y {\rm d}s \frac{{\rm d}R}{R} {\rm d}t \lesssim \nu,
\end{aligned}
\label{Me}
\end{equation*}
with $ \nu = \frac{R_0e^{J}}{JT_0} + \eps $, where $ 0 < \eps < 1, J \ge \eps^{-1} $ are both constant and
$ L^{\xi}_{\kappa_1\kappa_2\kappa_3} =  \kappa_1 \vert \nabla u^{1,\xi}(x) \vert^2 + \kappa_2 \vert \nabla u^{2,\xi}(x) \vert^2 + \kappa_3 \vert \nabla u^{3,\xi} \vert^2 $, $N_{\kappa_1 \kappa_2 \kappa_3 } = \kappa_2\kappa_3  \vert u^1(y) \vert^2 + \kappa_1\kappa_3\vert u^2(y) \vert^2$
$+ \kappa_1\kappa_2\vert u^3(y) \vert^2  $,
$ \chi\in C^{\infty} $ is a radial decreasing function satisfying
\begin{equation}
\chi(x) = \left\{
\begin{aligned}
1 & \quad \quad \vert x \vert \le 1 - \eps, \\
0 & \quad \quad \vert x \vert > 1,
\end{aligned}
\right.
\label{Gamma}
\end{equation}
and $ \b{\rm u}^{\xi} = (e^{ \frac{ix \cdot xi}{ \kappa_1 }}u^1, e^{ \frac{ix \cdot xi}{ \kappa_2 }}u^2, e^{ \frac{ix \cdot xi}{ \kappa_3 }}u^3) $ with
\begin{equation}
\xi = - \frac{ \kappa_1 \kappa_2 \kappa_3 \int_{\mathbb{R}^5} \Im \left( \overline{u^1(x)} \nabla u^1+ \overline{u^2(x)} \nabla u^2 +  \overline{u^3(x)} \nabla u^3 \right) \chi^{2} \left(\frac{x-s}{R}\right) {\rm d}x}{\int_{\mathbb{R}^5} \left(\kappa_2 \kappa_3 \vert u^1(x) \vert^2 + \kappa_1\kappa_3 \vert u^2(x) \vert^2  + \kappa_1\kappa_2 \vert u^3 \vert^2 \right) \chi^{2} \left(\frac{x-s}{R}\right){\rm d}x}
\label{xi}
\end{equation}
(unless the denominator is zero, in which case  $\xi = \xi(t, s, R) = \b{\rm 0 }  $).
\end{proposition}

Combining the two steps, that are respectively formulated as above, we can obtain Theorem \ref{main}.
\begin{remark}
{\bf Why we need $ \frac1{\kappa_3} = \frac1{\kappa_1} + \frac1{\kappa_2} $?}

That $ \mathcal{C + E} $ in \eqref{C+E} always stays invariance under Galilean transformation is important in the computation of $ \frac{\rm d}{{\rm d}t} M(t) $. 
Physically, $ \kappa_i \left(i=1,2,3\right)$ reflects the masses of the three particles in this system. 
The  mass-resonance condition \eqref{mr} is used to match the linear terms and nonlinear term in \eqref{NLS system}.
Only when the \eqref{NLS system} satisfies the mass-resonance \eqref{mr},   we can deduce the coercivity  estimates (Lemma \ref{coer2}, Lemma \ref{coerb}),
which is necessary to bound the major term of the interaction Morawetz estimate.
\end{remark}

\subsection{Outline of the paper}
The organization of this paper is as follows.
In Section 2, we clarify some preliminaries, especially, Strichartz estimates and well-posedness. In Section 3, we study the variational analysis of the ground state, and use this to obtain some coercivity estimates. Later, we prove the scattering criterion and interaction Morawetz estimate in Section 4 and 5, respectively. Finally, in Section 6, we use the results of Proposition \ref{decay} and \ref{scat} to complete the proof of main theorem.

\section{Preliminaries}
We mark $ A \lesssim B $ to mean there exists a constant $ C > 0 $ such that $ A \leqslant C B $.We indicate dependence on parameters via subscripts, e.g. $ A \lesssim_{x} B $ indicates $ A \leqslant CB $ for some $ C = C(x) > 0 $.

We write $ L_{t}^{q}L_{x}^{r} $ to denote the Banach space with norm
\begin{equation*}
\Vert u \Vert_{L_{t}^{q}(\mathbb{R}, L_{x}^{r}(\mathbb{R}^5))} := \left( \int_{\mathbb{R}} \left( \int_{\mathbb{R}^5} \vert u(t, x) \vert^{r} \mathrm{d}x \right) ^{\tfrac{q}{r}} \mathrm{d}t \right)^{\tfrac{1}{q}},
\end{equation*}
with the usual modifications when $ q $ or $ r $ are equal to infinity,
or when the domain $ \mathbb{R} \times \mathbb{R}^5 $ is replaced by space-time slab such as $ I \times \mathbb{R}^5 $.

Lastly, to fit this artical, we use the notation $ {\rm L}^{q} $ to denote $ L^{q} \times L^{q} \times L^{q} $ and $ {\rm H}^s $ to denote $ H^s \times H^s  \times H^s $.

\subsection{Some useful inequalities}\label{ineq}
In this subsection, we show some important inequalities which  will be used frequently in the following sections.

\begin{lemma}
	[Hardy-Littlewood-Sobolev inequality]\label{HLS}
	If $ 0 < \beta < 5 $, $ 1 < q < p < +\infty $ and
	\begin{equation*}
	\frac{1}{q} = \frac{1}{p} + 1 - \frac{\beta}{5},
	\end{equation*}
	then,
	\begin{equation*}
	\left\Vert \vert \cdot \vert^{-\beta} * f \right\Vert_{L^p(\mathbb{R}^5)} \le C \Vert f \Vert_{L^q(\mathbb{R}^5)}.
	\end{equation*}
\end{lemma}

We say the pair $(q,r)$ is $\dot{\rm H}^s$-admissible, if it  satisfies the condition
$$\frac{2}{q}=\frac{5}{2}-\frac{5}{r}-s,\ \ 2\leq q,r\leq\infty.$$
For $s\in(0,1)$, We define the sets $\Lambda_s$:
$$\Lambda_s:=\left\{(q,r)\ \text{is}\ \dot{H}^s\text{-admissible};\Big(\frac{10}{5-2s}\Big)^+\leq r\leq \Big(\frac{10}{3}\Big)^-\right\},$$
and
$$\Lambda_{-s}:=\left\{(q,r)\ \text{is}\ \dot{H}^{-s}\text{-admissible};\Big(\frac{10}{5-2s}\Big)^+\leq r\leq \Big(\frac{10}{3}\Big)^-\right\},$$
and let $\Lambda_0$ denote the $L^2$-admissible pairs. Here, $a^-$ is a fixed number slightly smaller than $a$ ($a^-=a-\epsilon$, where $\epsilon$ is small enough). $a^+$ can be defined by the similar way.

Next, we define the following Strichartz norm
$$\|{\b{\rm u}}\|_{S(\dot{\rm H}^s,I)}=\sup_{(q,r)\in \Lambda_s}\|{\b{\rm u}}\|_{L_t^q{\rm L}_x^r(I)}$$
and dual Strichartz norm
$$\|{\b{\rm u}}\|_{S^{'}(\dot{\rm H}^{-s},I)}=\inf_{(q,r)\in \Lambda_{-s}}\|{\b{\rm u}}\|_{L_t^{q^{'}}{\rm L}_x^{r^{'}}(I)},$$
where $q'$ denotes the dual exponent to $q$, i.e. the solution to $\tfrac1q+\frac1{q'}=1.$
 If $I=\R$, $I$ is omitted usually.

Now, let us recall the Strichartz estimates for free Schr\"odinger system, see \cite{CK, Fo, Ginibre1992, Keel1998, Strichartz1977} for more details.
\begin{lemma}\label{strichartz}
  Let $0\in I\subset\R$. The following statements hold
  \begin{enumerate}
    \item[$(i)$]$($Linear estimate$)$
    $$\|\mathcal{S}(t)\b{\rm g}\|_{S(\dot{\rm H}^s)}\leq C\|\b{\rm g}\|_{\dot{\rm H}^s};$$
    \item[$(ii)$]$($Nonlinear estimate I$)$
  $$\left\|\int_{0}^{t}\mathcal{S}(t-s)\b{\rm g}(\cdot,s)ds\right\|_{S({\rm L}^2,I)}\leq C\|\b{\rm g}\|_{S^{'}({\rm L}^2,I)};$$
    \item[$(iii)$]$($Nonlinear estimate II$)$
  $$\left\|\int_{0}^{t}\mathcal{S}(t-s) \b{\rm g}(\cdot,s)ds\right\|_{S(\dot{\rm H}^s,I)}\leq C\|\b{\rm g}\|_{S^{'}(\dot{\rm H}^{-s},I)}.$$

  \end{enumerate}
\end{lemma}

Using lemma \ref{strichartz}  and combining with standard contraction mapping theorem (see \cite{Ca}), we can obtain the local well-posedness theory for \eqref{NLS system}.
\begin{theorem}[Well-posedness]\label{well-posed}
  Let ${\b{\rm u}}_0\in \rm{H}^1$. Then there exists $T>0$ and a unique solution to \eqref{NLS system} satisfying
  $$\b{\rm u}\in C([0,T], {\rm{H}}_x^1(\R^5))\cap L_{\text{loc}}^q((0,T), {\rm{W}}_x^{1,r}(\R^5))).$$
  In particular, if $\b{\rm u}$ is bounded in $\rm{H}^1$ unniformly in $t$, then the solution $\b{\rm u}$ is global.
\end{theorem}

\section{Variational characterization}
In this section, we are in the position to prove the variational characterization for the sharp Gargliardo-Nirenberg inequality.
Firstly, we  will show the existence of the ground state.
As a consequence, we can obtain the sharp Gargliardo-Nirenberg inequality.
 And we use the properties of the ground state to establish the coercivity condition which will be used in  the proof of Morawetz estimate in the following section. 

First, we study the existence of ground state following the argument of Weinstein in \cite{Wein1982}. For other results, please see \cite{HOT}.
\begin{proposition}
[Existence of ground state]\label{ground state}
The minimal $J_{\min}$ of the nonnegative funtional
\[
    J(\b{\rm u}) : = (M(\b{\rm u}))^{\frac{1}{2}}(K(\b{\rm u}))^{\frac{5}{2}}(V(\b{\rm u}))^{-2},  \qquad  \b{\rm u} \in {\rm H}^{1}(\mathbb R^{5}) \setminus \{{\b 0}\}
\]
are attained at $ {\b{\rm u}} = (u^1, u^2, u^3)^T $
whose expression has to be in the form of $ (u^1, u^2, u^3)^T\\
= (e^{i\theta_1} m \phi_1(nx), e^{i\theta_2} m \phi_2(nx), e^{i\theta_3} m\phi_3(nx))^T $,
where $ m, n > 0 $, $ \theta_1, \theta_2, \theta_3 \in \mathbb R $,
and $ \b{\rm Q} = (\phi_1, \phi_2, \phi_3)^T \neq {\b 0} $ is the non-negative radial solution of the elliptic equation system:
\[
\left\{
\begin{aligned} \label{GS}
& \phi_1- \kappa_1\Delta \phi_1 = \phi_2 \phi_3, \\
& \phi_2- \kappa_2\Delta \phi_2 = \phi_1 \phi_3, \\
& 2\phi_3- \kappa_3\Delta \phi_3 = \phi_1 \phi_2
\end{aligned}
\right.
\qquad x \in \mathbb{R}^5.
\]
where the $ {\b{\rm Q}} $ is the {\bf ground state} with $ J({\b{\rm Q}}) = J_{\min} $. The sets of all ground states are denoted as $ \mathcal{G} $.
All ground states share the same mass, denoted as $ M_{gs} $.
\end{proposition}

\begin{proof}
	Suppose that the non-zero function sequence $\{\mathbf{u}_n\}$ is the minimal sequence of the functional $J$, i.e.
	$$
	\lim_{n\to\infty} J(\mathbf{u}_n)=\inf\{J(\mathbf{u}): \mathbf{u}\in {\rm H}^1(\R^5)\backslash\{\mathbf{0}\}\}.
	$$
	By the Schwartz symmetrical rearrangement lemma, without loss of generality, we assume $\mathbf{u}_n$ are non-negative radial.
	
	\textbf{Step.1} Note that for any $\mathbf{u}\in {\rm H}^1(\R^5)\backslash\{\mathbf{0}\}, \mu,\nu>0$, we have
	\begin{equation}\label{mhr}
	\left\{
	\begin{aligned}
	&M(\mu\mathbf{u}(\nu\cdot))=\mu^2\nu^{-5}M(\mathbf{u}), \\
	&K(\mu\mathbf{u}(\nu\cdot))=\mu^2\nu^{2-5}K(\mathbf{u}), \\
	&V(\mu\mathbf{u}(\nu\cdot))=\mu^3\nu^{-5}V(\mathbf{u}).
	\end{aligned}
	\right.
	\end{equation}
	By the defination of $J(\mathbf{u})$, it is easy to see $J(\mu\mathbf{u}(\nu\cdot))=J(\mathbf{u})$. Denote
	$$
	\mathbf{v}_n(x)=\frac{[M(\mathbf{u}_n)]^{\frac{5-2}{4}}}{[K(\mathbf{u}_n)]^{\frac 54}}\mathbf{u}_n\left( \left(\frac{M(\mathbf{u}_n)}{K(\mathbf{u}_n)} \right) ^{\frac12}x\right).
	$$
	On one hand, the mass and kinetic energy of $\mathbf{v}_n$ have been unitized, i.e.
	$$
	M(\mathbf{v}_n)=K(\mathbf{v}_n)\equiv 1.
	$$
	On the other hand, $ \mathbf{v}_n $ also retain non-negative radial at the same tine and satisfies
	$ J(\mathbf{v}_n)=J(\mathbf{u}_n) $, which implies
	$$
	\lim_{n\to\infty} J(\mathbf{v}_n)=\inf\{J(\mathbf{u}): \mathbf{u} \in {\rm H}^1(\R^5) \backslash \{\mathbf{0}\} \}.
	$$
	
	We note that $ \{ \mathbf{v}_n \}_{n=1}^\infty \in {\rm H}^1_{\rm rad}(\R^5) $ is bounded and the embedding $ H^1_{\rm rad}(\R^5)\to L^3(\R^5)$ is compact,
	then there exists a subsequence $ \{ \mathbf{v}_{n_k} \}_{k=1}^\infty \subset \{ \mathbf{v}_n \}_{n=1}^\infty $ and $\mathbf{v}^{\star}\in {\rm H}^1_{\rm rad} $, such that as $k\to \infty$, we have $\mathbf{v}_{n_k}\rightharpoonup\mathbf{v}^{\star}$ in ${\rm H}^1_{\rm rad}(\R^5)$ and $\mathbf{v}_{n_k}\to\mathbf{v}^{\star}$ in ${\rm L}^3(\R^5)$.
	
	By thw weak low semi-continuity of the functional $M$ and $K$, we have
	$$
	M(\mathbf{v}^{\star}), K(\mathbf{v}^{\star})\leq 1.
	$$
	Then using triangle inequality, we obtain
	$$
	\begin{aligned}
	V(\mathbf{v}_{n_k})&=\|\mathbf{v}_{n_k}\|_{{\rm L}^3(\R^5)}^3\\
	&\leq\left( \|\mathbf{v}_{n_k}-\mathbf{v}^{\star}\|_{{\rm L}^3(\R^5)}+\|\mathbf{v}^{\star}\|_{{\rm L}^3(\R^5)}\right)^3\\
	&\longrightarrow\|\mathbf{v}^{\star}\|_{{\rm L}^3(\R^5)}^3=V(\mathbf{v}^{\star}).
	\end{aligned}
	$$
Thus, we have that
$$
\begin{aligned}
J(\mathbf{v}^{\star})=&[M(\mathbf{v}^{\star})][K(\mathbf{v}^{\star})]^5[V(\mathbf{v}^{\star})]^{-4}
\leq\lim_{k\to\infty}[V(\mathbf{v}_{n_k})]^{-4}\\
=&\lim_{k\to\infty}J(\mathbf{v}_{n_k})=\inf\{J(\mathbf{u}): \mathbf{u} \in {\rm H}^1(\R^5) \backslash \{\mathbf{0}\} \},
\end{aligned}
$$
which shows that the minimal of $ J $ attains at $ \mathbf{v}^\star $. 	

{\bf Step 2.} We consider the variational derivatives of $M, K, V$: fix $\mathbf{u}=(u^1, u^2, u^3)^T\in {\rm H}^1(\R^5) \backslash \{\mathbf{0}\}$, for any $\mathbf{w}=(w_1, w_2, w_3 )^T\in {\rm H}^1(\R^5)$,
\begin{align}
\left.\frac{{\rm d}}{{\rm d}h}\right|_{h=0}M(\mathbf{u}+h\mathbf{w})=&\Re\int_{\R^5} u^1\overline{w_1}{\rm d}x+\Re\int_{\R^5}u^2\overline{w_2} {\rm d}x +\Re\int_{\R^5} 2u^3\overline{w_3}{\rm d}x \nonumber\\
=&\Re\int_{\R^5}(u^1, u^2, 2u^3 )\cdot\overline{(w_1, w_2, w_3 )}{\rm d}x\\
\left.\frac{{\rm d}}{{\rm d}h}\right|_{h=0}K(\mathbf{u}+h\mathbf{w})=&\Re\int_{\R^5}-\kappa_1\Delta u^1\overline{w_1}{\rm d}x+\Re\int_{\R^5} -\kappa_2\Delta u^2 \overline{w_2}{\rm d}x \\
&+\Re\int_{\R^5} -\kappa_3\Delta u^3 \overline{w_3}{\rm d}x \nonumber\\
=&\Re\int_{\R^5}(-\kappa_1 \Delta u^1, -\kappa_2 \Delta u^2, -\kappa_3 \Delta u^3)\cdot\overline{(w_1,w_2, w_3)}{\rm d}x\\
\left.\frac{{\rm d}}{{\rm d}h}\right|_{h=0}V(\mathbf{u}+h\mathbf{w})=&\Re\int_{\R^5}\bar{u^2}u^3\overline{w_1}{\rm d}x+\Re\int_{\R^5} \bar u^1 u^3 \overline{w_2}{\rm d}x+ Re\int_{\R^5}u^1 u^2 \overline{w_3}{\rm d}x\nonumber\\
=&\Re\int_{\R^5}(\bar u^2 u^3, \bar u^1 u^3, u^1 u^2)\cdot\overline{(w_1,w_2, w_3)}{\rm d}x.
\end{align}
If the functional $J$ attains the minimum at $\mathbf{W}=(\Phi_1,\Phi_2, \Phi_3)^T$, then we have for any $\mathbf{w}=(w_1,w_2, w_3)^T\in {\rm H}^1(\R^5)$,
$$
\left.\frac{{\rm d}}{{\rm d}h}\right|_{h=0}J(\mathbf{W}+h\mathbf{w})=0.
$$	
It means that
$$
\begin{aligned}
0=&[M(\Phi_1,\Phi_2, \Phi_3)]^{-1}\Re\int_{\R^5}(\Phi_1,\Phi_2, 2\Phi_3) \cdot\overline{(w_1,w_2, w_3)}{\rm d}x\\
&+5[K(\Phi_1,\Phi_2, \Phi_3)]^{-1}\Re\int_{\R^5}(-\kappa_1 \Delta\Phi_1,-\kappa2 \Delta\Phi_2, -\kappa_3 \Delta \Phi_3 )\cdot\overline{(w_1, w_2, w_3)}{\rm d}x\\
&-4[V(\Phi_1,\Phi_2, \Phi_3)]^{-1} \Re\int_{\R^5}(\Phi_2 \Phi_3, \Phi_1 \Phi_3, \Phi_1\Phi_2)\cdot\overline{(w_1,w_2, w_3)}{\rm d}x.
\end{aligned}
$$

By the arbitrariness of $ \mathbf{w} = (w_1, w_2, w_3)^T $, we have
$$
\left\{
\begin{aligned}
&[M(\Phi_1,\Phi_2, \Phi_3)]^{-1}\Phi_1- 5[K(\Phi_1,\Phi_2,\Phi_3)]^{-1} \kappa_1 \Delta\Phi_1-4[V(\Phi_1,\Phi_2, \Phi_3)]^{-1} \Phi_2 \Phi_3=0, \\
&[M(\Phi_1,\Phi_2,\Phi_3)]^{-1}\Phi_2-5[K(\Phi_1,\Phi_2, \Phi_3 )]^{-1}\kappa_2 \Delta\Phi_2 -4[V(\Phi_1,\Phi_2, \Phi_3 )]^{-1}\Phi_1\Phi_3=0,\\
&[M(\Phi_1,\Phi_2,\Phi_3)]^{-1}2\Phi_3-5[K(\Phi_1,\Phi_2, \Phi_3 )]^{-1}\kappa_3 \Delta\Phi_3 -4[V(\Phi_1,\Phi_2, \Phi_3 )]^{-1}\Phi_1\Phi_2=0.
\end{aligned}
\right.
$$
The above is equivalent to
$$
\left\{
\begin{aligned}
&\alpha \Phi_1- \kappa_1 \Delta\Phi_1= \beta\Phi_2\Phi_3, \\
&\alpha \Phi_2 -\kappa_2\Delta\Phi_2 =\beta\Phi_1\Phi_3,\\
& \alpha \Phi_3- \kappa_3 \Delta \Phi_3 =\beta\Phi_1\Phi_2
\end{aligned}
\right.
$$
where
$$
\alpha = \frac{[K(\Phi_1,\Phi_2,\Phi_3)]}{5[M(\Phi_1,\Phi_2, \Phi_3)]},
\beta = \frac{4[K(\Phi_1,\Phi_2,\Phi_3)]}{5[V(\Phi_1,\Phi_2, \Phi_3)]}.
$$
 Set $ \mathbf{W}(x) = \frac{\alpha}{\beta} \mathbf{Q}(\sqrt{\alpha} x) $ i.e. $ (\Phi_1(x), \Phi_2(x), \Phi_3(x)) = (\frac{\alpha}{\beta} \phi_1(\sqrt{\alpha} x), \frac{\alpha}{\beta} \phi_2 (\sqrt{\alpha} x), \frac{\alpha}{\beta} \phi_3 (\sqrt{\alpha} x) ) $, we obtain that
$$
\left\{
\begin{aligned}
&\phi_1- \kappa_1 \Delta\phi_1 =\phi_2\phi_3, \\
&\phi_2 -\kappa_2 \Delta\phi_2 =\phi_1\phi_3,\\
&2\phi_3 -\kappa_3 \Delta \phi_3 = \phi_1 \phi_2.
\end{aligned}
\right.
$$
The non-negative radial solutions of the above equation are the ground states. And let $\mathcal{G}$ denote the set of all ground states.

{\bf Step 3.} We show that the ground state has to be radial.
We note the fact that the radiality of $ \mathbf{Q} $ and $ \mathbf{W} $ are equivalent.
If $\mathbf{W}$ is non-radial, then by the Schwartz symmetrical rearrangement lemma, $M(\mathbf{W}^{\star})=M(\mathbf{W})$,  $K(\mathbf{W}^{\star})<K(\mathbf{W})$,  and $V(\mathbf{W}^{\star})\geq  V(\mathbf{W})$, so we have $J(\mathbf{W}^{\star})<J(\mathbf{W})$, which is contradict to the minimality of $\mathbf{W}$.
Thus $\mathbf{W}$ is radial and so is $ \mathbf{Q} $.

{\bf Step 4.} We show that if $\mathbf{W}=(\Phi_1,\Phi_2, \Phi_3)^T$ is the minimal element, then there exists three constants $\theta_1, \theta_2, \theta_3\in\R$ such that $(\Phi_1(x),\Phi_2(x),\Phi_3(x))^T=(e^{i\theta_1}m\phi_1(nx),
e^{i\theta_2}m\phi_2(nx),$ $e^{i\theta_3}m\phi_3(nx) )^T$ with $m=\frac{\alpha}{\beta},n=\sqrt{\alpha}$.

Since $J(m\phi_1(nx), m\phi_2(nx), m\phi_3(nx))\leq J(\Phi_1(x),\Phi_2(x), \Phi_3(x)) $, it is easy to find $ (\phi_1,\phi_2, \phi_3)^T$ is also a minimal element. Suppose that $(\Phi_1(x),\Phi_2(x),\Phi_3(x))=(e^{i\theta_1(x)}m\phi_1(nx),e^{i\theta_2(x)}m\phi_2(nx)$, $e^{i\theta_3(x)}m\phi_3(nx) )$, where $\theta_1(x),\theta_2(x),
\theta_3(x)$ are real-valued functions, then
$$
\begin{aligned}
|\nabla\Phi_1(x)|^2=&|\nabla e^{i\theta_1(x)}m\phi_1(nx)|^2\\
=&|i\nabla\theta_1(x)e^{i\theta_1(x)}m\phi_1(nx)+e^{i\theta_1(x)}mn\nabla \phi_1(nx)|^2\\
=&|i\nabla\theta_1(x)m\phi_1(nx)+mn\nabla \phi_1(nx)|^2\\
=&|mn\nabla\phi_1(nx)|^2+|\nabla\theta_1(x)m\phi_1(nx)|^2,
\end{aligned}
$$

$$
\begin{aligned}
|\nabla\Phi_2(x)|^2=&|\nabla e^{i\theta_2(x)}m\phi_2(nx)|^2\\
=&|i\nabla\theta_2(x)e^{i\theta_2(x)}m\phi_2(nx)+e^{i\theta_2(x)}mn\nabla\phi_2(nx)|^2\\
=&|i\nabla\theta_2(x)m\phi_2(nx)+mn\nabla \phi_2(nx)|^2\\
=&|mn\nabla\phi_2(nx)|^2+|\nabla\theta_2(x)m\phi_2(nx)|^2,
\end{aligned}
$$
and
$$
\begin{aligned}
	|\nabla\Phi_3(x)|^2=&|\nabla e^{i\theta_3(x)}m\phi_3(nx)|^2\\
	=&|i\nabla\theta_3(x)e^{i\theta_3(x)}m\phi_3(nx)+e^{i\theta_3(x)}mn\nabla\phi_3(nx)|^2\\
	=&|i\nabla\theta_3(x)m\phi_3(nx)+mn\nabla \phi_3(nx)|^2\\
	=&|mn\nabla\phi_3(nx)|^2+|\nabla\theta_3(x)m\phi_3(nx)|^2.
\end{aligned}
$$

By the minimality of $J(\Phi_1(x),\Phi_2(x), \Phi_3(x) ), J(\Phi_1(x),\Phi_2(x), \Phi_3(x))=J(m\phi_1(nx)$, $m\phi_2(nx), m\phi_3(nx))$. But by $M(\Phi_1(x),\Phi_2(x), \Phi_3)=M(m\phi_1(nx),m\phi_2(nx),m\phi_3(nx))$ and $V(\Phi_1(x),\Phi_2(x), \Phi_3(x))=V(m\phi_1(nx),m\phi_2(nx), m\phi_3(nx))$, so we have $K(\Phi_1(x)$,
$\Phi_2(x),\Phi_3(x))$ = $K(m\phi_1(nx),m\phi_2(nx), m\phi_3(nx))$. Then $\nabla\theta_1(x)=\nabla\theta_2(x)=\nabla\theta_3(x)=0$, thus $\theta_1(x),\theta_2(x), \theta_3(x)$ are both constants.
Therefore,
\begin{equation}
(\Phi_1(x),\Phi_2(x), \Phi_3(x))= (e^{i\theta_1}m\phi_1(nx),e^{i\theta_2}m\phi_2(nx), e^{i\theta_3}m\phi_3(nx)),
\label{W}
\end{equation}
where $m,n>0, \theta_1,\theta_2, \theta_3\in\R$ and $(\phi_1,\phi_2,\phi_3)^T\neq \b{\rm 0} $ is the non-negative non-zero radial solution to \eqref{GS}.
\end{proof}

Similarly to the idea in our previous work \cite{Meng2020}, we can also compute the ratio of $M(\b{\rm Q})$, $K(\b{\rm Q})$ and $V(\b{\rm Q}) $, which is helpful for the optimal constant in Gargliardo-Nirenberg inequality.
\begin{remark}
For ground state $ \b{\rm Q} = (\phi_1, \phi_2, \phi_3)^T $, we have the following scaling identity
\begin{align*}
& M \left( \lambda^{\alpha} \b{\rm Q}(\lambda^{\beta}\cdot) \right) + E \left( \lambda^{\alpha} \b{\rm Q}(\lambda^{\beta}\cdot) \right) \\
= & \lambda^{2\alpha - 5\beta}M(\b{\rm Q}) + \lambda^{2\alpha - 3\beta}K(\b{\rm Q}) - \lambda^{3\alpha - 5\beta}V(\b{\rm Q}) \quad \forall \lambda \in (0, \infty).
\end{align*}
Using variational derivatives and letting  $ \lambda = 1 $ in both sides of above identity, we can obtain
\[
\begin{aligned}
0 = & \Re \int_{\mathbb{R}^5} \left(  \phi_1- \kappa_1\Delta \phi_1 - \phi_2 \phi_3 \right) \cdot \overline{ \left( \frac{\rm d}{{\rm d}\lambda} \Big|_{\lambda=1} \lambda^{\alpha} \phi_1(\lambda^{\beta}x) \right) } \\
& \quad \quad + \left(  \phi_2- \kappa_2\Delta \phi_2 - \phi_1 \phi_3 \right) \cdot \overline{ \left( \frac{\rm d}{{\rm d}\lambda} \Big|_{\lambda=1} \lambda^{\alpha} \phi_2(\lambda^{\beta}x) \right) }\\
&\left( 2 \phi_3- \kappa_3\Delta \phi_3 -\phi_1 \phi_2 \right) \cdot \overline{ \left( \frac{\rm d}{{\rm d}\lambda} \Big|_{\lambda=1} \lambda^{\alpha} \phi_3(\lambda^{\beta}x) \right)} {\rm d}x   \\
= & (2\alpha - 5\beta)M(\b{\rm Q}) + (2\alpha - 3\beta)K(\b{\rm Q}) - (3\alpha - 5\beta)V(\b{\rm Q}) \\
= & (2M + 2K - 3V)\alpha - (5M + 3K -5V)\beta, \quad \forall \alpha, \beta \in \mathbb{R}.
\end{aligned}
\]
This yields
\begin{equation}
M(\b{\rm Q}) : K(\b{\rm Q}) : V(\b{\rm Q}) = 1 : 5 : 4.
\label{Q}
\end{equation}
\end{remark}

Using the properties of the ground state, we can directly obtain the sharp Gagliardo--Nirenberg inequality.
\begin{corollary}[Gagliardo--Nirenberg inequality]
For any $\b{\rm u}\in {\rm H}^1\setminus\{\b{\rm 0}\}$, we have
\begin{equation}
V(\b{\rm u}) \leqslant C_{GN} [M(\b{\rm u})]^{\frac{1}{4}} [K(\b{\rm u})]^{\frac{5}{4}},
\label{GN-inequality}
\end{equation}
where $ C_{GN} = 4 \cdot 5^{-\frac{5}{4}} M_{gs}^{-\frac{1}{2}} $.
The equality holds if and only if $ \b{\rm u} \in {\rm H}^{1}(\mathbb R^{5}) \in \mathcal G  $.
\end{corollary}

Thanks to the sharp Gagliardo--Nirenberg, Pohozaev identities and conservation of mass and energy, we have the following coercivity. We will omit the proof since it is similar to the argument in \cite{Dodson2016,Dodson2018}
\begin{lemma}[Coercivity $I$]\label{coer1}
For the \eqref{NLS system} under the mass-resonance condition \eqref{mr}.
If $ M(\b{\rm u}_0) E(\b{\rm u}_0) \le (1 - \delta) M(\b{\rm Q}) E(\b{\rm Q}) $
and $ M(\b{\rm u}_0) K(\b{\rm u}_0) < M(\b{\rm Q}) K(\b{\rm Q}) $, then there exists $ \delta^\prime = \delta^\prime (\delta) > 0 $ so that
\begin{equation}
M(\b{\rm u}) K(\b{\rm u}) \le (1 - \delta^\prime) M(\b{\rm Q}) K(\b{\rm Q})
\label{coe1}
\end{equation}
for all $ t \in I $,  where $ \b{\rm u} : I \times \mathbb{R}^5 \to \mathbb{C}^3 $ is the maximal-lifespan solution to \eqref{NLS system}.
\end{lemma}
\begin{remark}
From \eqref{coe1} and the mass conservation, we have that $ \b{\rm u} $ is uniformly bounded in $ {\rm H}^1(\mathbb{R}^5)$. By Theorem \ref{well-posed}, the solution to \eqref{NLS system} $ \b{\rm u} $ is global.
\end{remark}

\begin{lemma}
[Coercivity $II$]\label{coer2}
Suppose $ M(\b{\rm u}) K(\b{\rm u}) \le (1 - \delta) M(\b{\rm Q}) K(\b{\rm Q}) $,
then there exists $ \delta^\prime = \delta^\prime (\delta) > 0 $ such that
\begin{equation}
4 K(\b{\rm u}^{\xi}) - 5 V(\b{\rm u}) \ge \delta^{\prime} K(\b{\rm u}^{\xi}),
\label{coe2}
\end{equation}
where $ \b{\rm u}^{\xi} $ be as in Proposition \ref{decay}.
\end{lemma}

\begin{proof}
Using the fact that $ C_{GN} = 4 \cdot 5^{-\frac{5}{4}} M_{gs}^{-\frac{1}{2}} = 4 \cdot 5^{-\frac{5}{4}} \left[ M(\b{\rm Q}) \right]^{-\frac{1}{2}} = \frac{4}{5} \left[ M(\b{\rm Q}) K(\b{\rm Q}) \right]^{-\frac{1}{4}} $, we have
\begin{equation*}
\begin{aligned}
V(\b{\rm u}) \le C_{GN} \left[ M(\b{\rm u}) \right]^{\frac{1}{4}} \left[ K(\b{\rm u}) \right]^{\frac{5}{4}}
= \frac{4}{5} \left[ \frac{M(\b{\rm u}) K(\b{\rm u}}{M({\b{\rm Q})} K(\b{\rm Q})} \right]^{\frac{1}{4}} K(\b{\rm u}).
\end{aligned}
\end{equation*}
Owing to $ M(\b{\rm u}) = M(\b{\rm u}^{\xi}) $ and $ V(\b{\rm u}) = V(\b{\rm u}^{\xi}) $, we know furthermore
\begin{equation*}
\begin{aligned}
V(\b{\rm u}) = V(\b{\rm u}^{\xi}) \le & \frac{4}{5} \inf_{\xi \in \mathbb{R}^5} \left\{ \left[ \frac{M(\b{\rm u}) K(\b{\rm u}^{\xi})}{M({\b{\rm Q})} K(\b{\rm Q})} \right]^{\frac{1}{4}} K(\b{\rm u}^{\xi}) \right\} \\
= & \frac{4}{5} \inf_{\xi \in \mathbb{R}^5} \left[ \frac{M(\b{\rm u}) K(\b{\rm u}^{\xi})}{M({\b{\rm Q})} K(\b{\rm Q})} \right]^{\frac{1}{4}} \inf_{\xi \in \mathbb{R}^5} K(\b{\rm u}^{\xi}) \\
\le & \frac{4}{5} (1 - \delta)^{\frac{1}{4}} K(\b{\rm u}^{\xi}),
\end{aligned}
\end{equation*}
which implies \eqref{coe2}.
\end{proof}

\begin{lemma}
[Coercivity on balls]\label{coerb}
There exists $ R = R(\delta,  M(\b{\rm u}),  \b{\rm Q}) > 0 $ sufficiently large such that
\begin{equation}
\sup_{t \in \mathbb{R}} M(\b{\rm u}^{\xi}_R) K(\b{\rm u}^{\xi}_R) < (1 - \delta) M(\b{\rm Q}) K(\b{\rm Q}),
\label{coe}
\end{equation}
where $ \b{\rm u}^{\xi}_R(x) = \chi_R(x) \b{\rm u}^{\xi}(x) $ for $ \chi_R $, a smooth cut function on $ B(0, R) \subset \mathbb{R}^5 $.
In particular, by Lemma \ref{coer2}, there exists $ \delta^\prime = \delta^\prime(\delta) > 0 $ so that
\begin{equation}
4 K(\b{\rm u}^{\xi}_R) - 5 V(\b{\rm u}_R) \ge \delta^{\prime} K(\b{\rm u}^{\xi}_R)
\label{coe3}
\end{equation}
uniformly for $ t \in \mathbb{R} $.
\end{lemma}

\begin{proof}
First note that multiplication by $ \chi $ only decreases the $ L^{2}(\mathbb{R}^5) $-norm, that is
\begin{equation*}
M \left( \chi_{R} \b{\rm u}^{\xi}(t) \right) \leqslant M \left( \b{\rm u}^{\xi}(t) \right)
\end{equation*}
uniformly for $ t \in \mathbb{R} $. Thus, it suffices to consider the $ \dot{\rm H}^1(\mathbb{R}^5) $-norm.
For this, we will make use of the following identity:
\begin{equation*}
\int_{\mathbb{R}^5} \chi_{R}^2 \vert \nabla u^{\xi} \vert^2 {\rm d}x = \int_{\mathbb{R}^5} \vert \nabla ( \chi_{R} u^{\xi} ) \vert^2 + \chi_{R} \Delta (\chi_{R}) \vert u^{\xi} \vert^2 {\rm d}x,
\end{equation*}
which can be obtained by a direct computation. In particular,
\begin{equation*}
K(\b{\rm u}^{\xi}_R) \le K(\b{\rm u}^{\xi}) + \mathcal{O}\left( \frac{1}{R^2} M(\b{\rm u}) \right).
\end{equation*}
Choosing $ R \gg 1 $ sufficiently large depending on $ \delta,  M(\b{\rm u}) $ and $ \b{\rm Q} $, the result follows.
\end{proof}

\section{Proof of scattering criterion}
In this section, we will follow the strategy in \cite{Dodson2018} to prove the scattering criterion (Proposition \ref{scat}).
Roughly speaking, it states that if in any large window of time there always exists an interval large enough on which the scattering norm is very small,
then the solution of \eqref{NLS system} has to scatter.
\begin{proposition}[Scattering criterion]\label{scat}
Let $ \b{\rm u}_0, \b{\rm Q}, I $ be as in Theorem \ref{decay}, and suppose further that $ \Vert \b{\rm u}_0 \Vert_{{\rm H}^1(\mathbb{R}^5)} \lesssim E_0 $.
Let $ \b{\rm u} : \mathbb{R} \times \mathbb{R}^5 \to \mathbb{C} $ be the corresponding global solution to \eqref{NLS system}.
Suppose that $ \exists ~ t_0 \in I $ such that
\begin{equation}
\Vert \b{\rm u} \Vert_{{\rm L}_{t}^6([t_0-l, t_0], {\rm L}_{x}^3(\mathbb{R}^5))} \le \eps^{\frac{1}{24}},
\label{b1}
\end{equation}
where $ \eps = \eps(E_0) $ is sufficiently small, $l=\eps ^{-\frac{1}{4}}$ and $ T_0 = T_0(\eps, E_0) $ is large enough.
Then $ \b{\rm u} $ scatters forward in time.
\end{proposition}

Before showing the proof of Proposition \ref{scat}, we state a space-time norm bound of the effect of nonlinear term $ \b{\rm f}(\b{\rm u}) $ in finite time, under the assumption of \eqref{b1}.

\begin{lemma}
	\label{LM3.2}
	Let $ \b{\rm u} $ be as in Proposition \ref{scat} and suppose that \eqref{b1} holds.
	Then $ \forall ~ a \in \mathbb{R}, \exists ~ t_0 \in [a, a+T_0] $ such that
	\begin{equation}
	\left\| \int_{0}^{t_0}\mathcal{S}(t-s) \b{\rm f}(\b{\rm u}(s)) {\rm d}s \right\|_{L_t^6([t_0,\infty),L_x^3(\mathbb{R}))} \lesssim \eps^{\frac1{24}}.
	\label{18}
	\end{equation}
\end{lemma}

\begin{proof}
	Let $ a \in \mathbb{R} $ and choose $ t_0 $ as in \eqref{b1}.
	Using Dumahel formula, Strichartz estimates, and H\"older inequality,
	\begin{align*}
	& ~ \left\| |\nabla|^\frac12 \b{\rm u} \right\|_{L_t^2([t_0-l,t_0),L_x^{\frac{10}{3}}(\mathbb{R}^5))} \\
	\le & ~ \left\| |\nabla|^\frac12 \b{\rm u}_0 \right\|_{L^2(\mathbb{R}^5)}
	+ \left\| |\nabla|^\frac12 \int_{0}^{t} \mathcal{S}(t-s) \b{\rm f}(\b{\rm u}(s)) {\rm d}s \right\|_{L_t^2([t_0-l,t_0),L_x^{\frac{10}{3}}(\mathbb{R}^5))} \\
	\le & ~ \left\| \b{\rm u}_0 \right\|_{{\dot H}^{\frac12}(\mathbb{R}^5)}
	+ \left\| |\nabla|^\frac12 \b{\rm f}(\b{\rm u}) \right\|_{L_t^{\frac32}([t_0-l,t_0),L_x^{\frac{30}{19}}(\mathbb{R}^5))} \\
	\le & ~ \left\| \b{\rm u}_0 \right\|_{H^1(\mathbb{R}^5)}
	+ \left\| \b{\rm u} \right\|_{L_t^6([t_0-l,t_0),L_x^3(\mathbb{R}^5))}
	\left\| |\nabla|^\frac12 \b{\rm u} \right\|_{L_t^2([t_0-l,t_0),L_x^\frac{10}{3}(\mathbb{R}^5))} \\
	\le & ~ \left\| \b{\rm u}_0 \right\|_{H^1(\mathbb{R}^5)} + \eps^{\frac1{24}} \left\| |\nabla|^\frac12 \b{\rm u} \right\|_{L_t^2([t_0-l,t_0),L_x^\frac{10}{3}(\mathbb{R}^5))}.
	\end{align*}
	By the standard bootstrap arguments, we can get
	$$ \left\| |\nabla|^\frac12 \b{\rm u} \right\|_{L_t^2([t_0-l,t_0),L_x^\frac{10}{3}(\mathbb{R}^5))} \lesssim 1. $$
	Thus, by Strichartz estimates, for any $ l \in (0, t_0) $
	\begin{align*}
	\left\| \int_{t_0-l}^{t_0} \mathcal{S}(t - s) \b{\rm f}(\b{\rm u}(s)) {\rm d}s \right\|_{L_t^6([t_0,\infty),L_x^3(\mathbb{R}^5))}
	\le & ~ \left\| |\nabla|^\frac12 \b{\rm f}(\b{\rm u}) \right\|_{L_t^{\frac32}([t_0-l,t_0),L_x^{\frac{30}{19}}(\mathbb{R}^5))} \\
	\lesssim & ~ \left\| \b{\rm u} \right\|_{L_t^6([t_0-l,t_0),L_x^3(\mathbb{R}^5))} \left\| |\nabla|^\frac12 \b{\rm u} \right\|_{L_t^2([t_0-l,t_0),L_x^\frac{10}{3}(\mathbb{R}^5))} \\
	\lesssim & ~ \eps^{\frac1{24}}.
	\end{align*}
	
	We check the remaining part of the integral by interpolation.
	Firstly, using the dispersive estimates and Sobolev embedding $ H^1(\mathbb{R}^5) \hookrightarrow L^{\frac{12}{5}}(\mathbb{R}^5) $, we obtain
	\begin{align*}
	\left\| \mathcal{S}(t - s) \b{\rm f}(\b{\rm u}(s)) \right\|_{L_x^6(\mathbb{R}^5)}
	& ~ \le |t - s|^{-\frac53} \left\| \b{\rm u}(s) \right\|_{L_x^{\frac{12}5}(\mathbb{R}^5)}^2 
	\lesssim |t - s|^{-\frac53} E_0^2.
	\end{align*}
	Thus,
	$$ \left\| \int_{0}^{t_0-l} \mathcal{S}(t - s) \b{\rm f}(\b{\rm u}(s)) {\rm d}s \right\|_{L_t^3([t_0,\infty),L_x^6(\mathbb{R}^5))}
	\lesssim E_0^2 ~ l^{-\frac13}. $$
	Next, using the identity
	\begin{equation*}
	\begin{aligned}
	i\int_{0}^{t_0-l} \mathcal{S}(t - s) \b{\rm f}(\b{\rm u}(s)) {\rm d}s
	= & ~ \mathcal{S}(t - t_0 + l) \left[ \b{\rm u}(t_0-l) - \mathcal{S}(t_0 - l) \b{\rm u}_0 \right] \\
	= & ~ \mathcal{S}(t-t_0+l) \b{\rm u}(t_0-l) - \mathcal{S}(t) \b{\rm u}_0,
	\end{aligned}
	\end{equation*}
	and Strichartz estimates, we have
	\begin{equation*}
	\left\Vert \mathcal{S}(t-t_0+l) \b{\rm u}(t_0-l) - \mathcal{S}(t) \b{\rm u}_0  \right\Vert_{L_t^\infty([t_0,\infty),L_x^2(\mathbb{R}^5))}
	\lesssim E_0.
	\end{equation*}
	Thus, by interpolation, we have
	\begin{equation*}
	\begin{aligned}
	& ~ \left\|\int_{0}^{t_0 - l}\mathcal{S}(t-s)\b{\rm f}(\b{\rm u}(s)) {\rm d}s \right\|_{L_t^6([t_0,\infty),L_x^3(\mathbb{R}^5))} \\
	\le & ~ \left\|\int_{0}^{t_0 - l}\mathcal{S}(t-s)\b{\rm f}(\b{\rm u}(s)) {\rm d}s \right\|_{L_t^3([t_0,\infty),L_x^6)}^{\frac12}
	\left\|\int_{0}^{t_0 - l}\mathcal{S}(t-s)\b{\rm f}(\b{\rm u}(s)) {\rm d}s \right\|_{L_t^\infty([t_0,\infty),L_x^2)}^{\frac12} \\
	\lesssim&_{E_0} ~ l^{-\frac16} = \eps^{\frac1{24}},
	\end{aligned}
	\end{equation*}
	if we are able to fix $ l = \eps^{-\frac14} \in (0, t_0) $.
	And this can be done easily, as long as we choose $ t_0 \in [a, a+T_0] $ to be large enough.
	
	To sum up,
	\begin{equation*}
	\int_{0}^{t_0}\mathcal{S}(t-s)\b{\rm f}(\b{\rm u}(s)) {\rm d}s
	= \int_{0}^{t_0 - l}\mathcal{S}(t-s)\b{\rm f}(\b{\rm u}(s)) {\rm d}s +
	\int_{t_0 - l}^{t_0}\mathcal{S}(t-s)\b{\rm f}(\b{\rm u}(s)) {\rm d}s,
	\end{equation*}
	so \eqref{18} holds.
\end{proof}

Now, we turn to prove the scattering criterion, Proposition \ref{scat}.

\begin{proof}
	First of all, we split $ \mathbb{R} $ into three intervals, $ \mathbb{R} = (-\infty, -T_0]\cup[-T_0,T_0]\cup[T_0, +\infty) $, which are denoted by $ I_1, ~ I_2, ~ I_3 $.
	We can choose $ T_0 \gg 1 $ large enough so that
	\begin{align}\label{b0}
	\left\| \mathcal{S}(t) \b{\rm u}_0 \right\|_{L_t^6((I_j),L_x^3(\mathbb{R}^5))} \lesssim \eps^{\frac1{24}}, \quad \forall ~ j = 1, ~ 3.
	\end{align}
	Our goal is to prove for $T_0$ large, we have
	$$ \| \b{\rm u} \|_{L_t^6((\mathbb{R}),L_x^3(\mathbb{R}^5))} \lesssim_{E_0} T_0. $$
	We only need to prove
	\begin{align}\label{b3}
	\| \b{\rm u} \|_{L_t^6((I_j),L_x^3(\mathbb{R}^5))}\lesssim_{E_0} T_0,
	\end{align}
	for each $j=1,\ 2,\ 3 $.
	Using Sobolev embedding and H\"older inequality
	\begin{align}\label{b2}
	\| \b{\rm u} \|_{L_t^6((I_2),L_x^3(\mathbb{R}^5))}^6 \lesssim_{E_0} T_0.
	\end{align}
	It is suffice to consider $j=1,\ 3$. Since the way of dealing with $I_1$ is same as in $I_3$, therefor we only check \eqref{b3} with $I_3$.
	
	Since $ a \in \mathbb{R} $ in Lemma \ref{LM3.2} can be taken arbitrarily, we take $ a = T_0 $. Thus, $ t_0 \in [a, a+T_0] = [T_0, 2T_0] $. Similarly to \eqref{b2}, we have
	\begin{equation*}
	\begin{aligned}
	\left\Vert \b{\rm u} \right\Vert^6_{{\rm L}_{t}^6(I_{2}, {\rm L}_{x}^3(\mathbb{R}^5))} = & \left\Vert \b{\rm u} \right\Vert^6_{{\rm L}_{t}^6((T_0, t_0], {\rm L}_{x}^3(\mathbb{R}^5))} + \left\Vert \b{\rm u} \right\Vert^6_{{\rm L}_{t}^6((t_0, \infty), {\rm L}_{x}^3(\mathbb{R}^5))} \\
	\lesssim & ~ T_0 + \left\Vert \b{\rm u} \right\Vert^6_{{\rm L}_{t}^6((t_0, \infty), {\rm L}_{x}^3(\mathbb{R}^5))}.
	\end{aligned}
	\end{equation*}
	
	We use dispersive estimates, H\"older inequality and Hardy-Littlewood-Sobolev inequality (Lemma \ref{HLS}) to find
	\begin{equation*}
	\begin{aligned}
	& ~ \left\Vert \b{\rm u} \right\Vert_{{\rm L}_{t}^6((t_0, \infty), {\rm L}_{x}^3(\mathbb{R}^5))} \\
	\le & ~ \left\Vert \mathcal{S}(t-t_0) \b{\rm u}(t_0) \right\Vert_{{\rm L}_{t}^6((t_0, \infty), {\rm L}_{x}^3(\mathbb{R}^5))} + \left\Vert \int_{t_0}^{t}|t-s|^{-\frac56} \|\b{\rm f}(\b{\rm u}(s))\|_{L_x^\frac32(\mathbb{R}^5)} {\rm d}s \right\Vert_{{\rm L}_{t}^6((t_0, \infty))} \\
	\le & ~ \left\Vert \mathcal{S}(t-t_0) \b{\rm u}(t_0) \right\Vert_{{\rm L}_{t}^6((t_0, \infty), {\rm L}_{x}^3(\mathbb{R}^5))} + \left\Vert \int_{t_0}^{t}|t-s|^{-\frac56} \|\b{\rm u}(s)\|_{{\rm L}_x^3(\mathbb{R}^5)}^2 {\rm d}s \right\Vert_{{\rm L}_{t}^6((t_0, \infty))} \\
	\le & ~ \left\Vert \mathcal{S}(t-t_0) \b{\rm u}(t_0) \right\Vert_{{\rm L}_{t}^6((t_0, \infty), {\rm L}_{x}^3(\mathbb{R}^5))} + \left\Vert \|\b{\rm u}(t)\|_{{\rm L}_{x}^3(\mathbb{R}^5)}^2 \right\Vert_{{\rm L}_t^3(t_0, \infty)}\\
	\le & ~ \left\Vert \mathcal{S}(t-t_0) \b{\rm u}(t_0) \right\Vert_{{\rm L}_{t}^6((t_0, \infty), {\rm L}_{x}^3(\mathbb{R}^5))} + \|\b{\rm u}\|_{{\rm L}_{t}^6((t_0, \infty), {\rm L}_{x}^3(\mathbb{R}^5))}^2.
	\end{aligned}
	\end{equation*}
	
	The continuity argument tells us if $ \left\Vert \mathcal{S}(t-t_0) \b{\rm u}(t_0) \right\Vert_{{\rm L}_{t}^6((t_0, \infty), {\rm L}_{x}^3(\mathbb{R}^5))} \le \eps^{\frac{1}{24}} $ then \eqref{b3} holds.
	
	Now, we turn to the following identity
	\begin{equation*}
	\mathcal{S}(t-t_0) \b{\rm u}(t_0) = \mathcal{S}(t) \b{\rm u}_0 + i \int_{0}^{t_0} \mathcal{S}(t-s) \b{\rm f}(\b{\rm u}(s)) {\rm d}s.
	\end{equation*}
	Combining this with \eqref{b0} then suffices to establish
	\begin{equation}
	\left\Vert \int_{0}^{t_0} {\mathcal S}(t-s)\b{\rm f}(\b{\rm u}(s)) {\rm d}s \right\Vert_{{\rm L}_{t}^6((t_0, \infty), {\rm L}_{x}^3(\mathbb{R}^5))} \lesssim_{E_0} \eps^{\frac{1}{24}}.
	\end{equation}
	Together with Lemma \ref{LM3.2}, which completes the proof.
\end{proof}

References related to this section are \cite{Dodson2016, Dodson2018, Meng20200, Meng2020, Tao2004}.

\section{Interaction Morawetz estimate}
In this part, we are now in the position to prove the interaction Morawetz estimate, Proposition \ref{Me}.

We start with a functional of $ \b{\rm u}(t, x) = (u^1(t, x), u^2(t, x), u^3(t,x))^T $, the solution of \eqref{NLS system}.
\begin{equation}
\begin{aligned}
M(t) = &  \int_{\mathbb{R}^5} M^a_y(t) N_{\kappa_1 \kappa_2 \kappa_3}  {\rm d}y,
\end{aligned}
\label{M}
\end{equation}
where
\[
M^a_y(t):=2\int_{\mathbb{R}^5} \Im \left( \sum_{i=1}^3 \overline{u^i(x)} \nabla u^i(x) \right) \cdot \nabla a(x-y) {\rm d}x, 
\]
\[
N_{\kappa_1 \kappa_2 \kappa_3} := \kappa_1 \kappa_2 \kappa_3 \sum_{i=1}^3 \frac{\vert u^i(y) \vert^2}{\kappa_i}, 
\]
and $ a \in C^{\infty} $ is a real function to be chosen later.

Let $ R \gg 1 $ be sufficiently large and let $ \phi $ and $ \phi_1 $ both be radial satisfying
\begin{equation*}
\phi(x) = \frac{1}{\omega_{5}R^{5}} \int_{\mathbb{R}^5} \chi^2 \left(\frac{x - s}{R}\right) \chi^2 \left(\frac{s}{R}\right) {\rm d}s,
\end{equation*}
and
\begin{equation*}
\phi_1(x) = \frac{1}{\omega_{5}R^{5}} \int_{\mathbb{R}^5} \chi^3 \left(\frac{x - s}{R}\right) \chi^2 \left(\frac{s}{R}\right) {\rm d}s,
\end{equation*}
where $ \omega_{5} $ is the volume of unit ball in $ \mathbb{R}^5 $ and $ \Gamma $ be as in \eqref{Gamma}.
Finally, we define
\begin{equation*}
\psi(x) = \frac{1}{\vert x \vert} \int_{0}^{\vert x \vert} \phi(r) {\rm d}r, \quad a(x) = \int_{0}^{\vert x \vert} \psi(r) r {\rm d}r.
\end{equation*}
\begin{proof}[Proof of Proposition \ref{decay}]
We rely on the equation \eqref{NLS system} to change $ \partial_t \b{\rm u} $ equally into the derivative of $ \b{\rm u} $ with respect to the space variable $ x $.
Then we have
\begin{equation}
\begin{aligned}
\frac{\mathrm{d}}{\mathrm{d}t}M(t)
= & \int_{\mathbb{R}^5} \partial_t M^a_y(t) N_{\kappa_1 \kappa_2 \kappa_3} {\rm d}y +\int_{\mathbb{R}^5} M^a_y(t) \partial_t N_{\kappa_1 \kappa_2 \kappa_3} {\rm d}y\\
=& - \int_{\mathbb{R}^5} \int_{\mathbb{R}^5} \Delta W_{\kappa_1 \kappa_2 \kappa_3} \Delta a(x-y) N_{\kappa} {\rm d}x {\rm d}y \\
& -2 \int_{\mathbb{R}^5} \int_{\mathbb{R}^5} \Re \left( \overline{u^1(x)u^2(x)} u^3(x) \right) \Delta a(x-y) N_{\kappa_1 \kappa_2 \kappa_3} {\rm d}x {\rm d}y \\
& +4 \sum_{k=1}^{5} \sum_{j=1}^{5} \int_{\mathbb{R}^5} \int_{\mathbb{R}^5} R^{jk}_{\kappa_1 \kappa_2 \kappa_3 } a_{jk}(x-y) N_{\kappa_1 \kappa_2 \kappa_3} {\rm d}x {\rm d}y \\
& -4 \kappa_1\kappa_2\kappa_3\sum_{k=1}^{5} \sum_{j=1}^{5} \int_{\mathbb{R}^5} \int_{\mathbb{R}^5} A^{j} a_{jk}(x-y) B^{k} {\rm d}x {\rm d}y,
\end{aligned}
\label{M1}
\end{equation}
where 
\begin{equation}
\begin{aligned}
& W_{\kappa_1 \kappa_2 \kappa_3} = \sum_{i=1}^3 \frac{\kappa_1 \kappa_2 \kappa_3}{\kappa_i} \vert u^i(x) \vert^2, \quad A^{j} = \Im \left( \sum_{i=1}^3 \overline{u^i(x)}u^i_j(x) \right), \\
& R^{jk}_{\kappa_1 \kappa_2 \kappa_3} = \Re \left( \sum_{i=1}^3 \kappa_i u^i_{k}(x) \overline{u^i_{j}(x)} \right), \quad B^{k} = \Im \left( \sum_{i=1}^3 \overline {u^i(y)} u^i_k(y) \right), 
\end{aligned}
\end{equation}
for $ j, k \in \{ 1, 2, 3, 4, 5 \} $ and we use $ \partial_l $ to denote the partial differential respected to $ x_l $ for $ l \in \{ 1, 2, 3, 4, 5 \} $.

Direct computations yield $ \Delta a = 4 \psi + \phi $ and $ a_{jk} = \delta_{jk} \phi + P_{jk} (\psi - \phi) $,
where $ P_{jk}(x) := \delta_{jk} - \frac{x_{j}x_{k}}{\vert x \vert^2} $ and $ \psi - \phi \ge 0 $.
Due to the facts above, we have
\begin{equation}
\begin{aligned}
\frac{\rm d}{{\rm d}t}M(t)
= & -\int_{\mathbb{R}^5} \int_{\mathbb{R}^5} \Delta W_{\kappa_1\kappa_2 \kappa_3} \left( 4 \psi(x-y) + \phi(x-y) \right) N_{\kappa_1 \kappa_2 \kappa_3} {\rm d}x {\rm d}y \\
& - 2 \int_{\mathbb{R}^5} \int_{\mathbb{R}^5} \Re \left( \overline{u^1(x)u^2(x)} u^3(x) \right) \left( 4 \psi(x-y) + \phi(x-y) \right) N_{\kappa_1 \kappa_2 \kappa_3} {\rm d}x {\rm d}y \\
& + 4 \int_{\mathbb{R}^5} \int_{\mathbb{R}^5} L_{\kappa_1\kappa_2\kappa_3} \phi(x-y) N_{\kappa_1 \kappa_2 \kappa_3} {\rm d}x {\rm d}y \\
& + 4 \sum_{k=1}^{5} \sum_{j=1}^{5} \int_{\mathbb{R}^5} \int_{\mathbb{R}^5} R^{jk}_{\kappa_1 \kappa_2 \kappa_3} P_{jk}(x-y) (\psi(x-y) - \phi(x-y)) N_{\kappa_1 \kappa_2 \kappa_3} {\rm d}x {\rm d}y \\
& - 4 \int_{\mathbb{R}^5} \int_{\mathbb{R}^5} A \phi(x-y) B {\rm d}x {\rm d}y \\
& - 4 \sum_{k=1}^{5} \sum_{j=1}^{5} \int_{\mathbb{R}^5} \int_{\mathbb{R}^5} A^{j} P_{jk}(x-y) (\psi(x-y) - \phi(x-y)) B^{k} {\rm d}x {\rm d}y \\
=: & \mathcal{A+B+C+D+E+F},
\end{aligned}
\end{equation}

where
\begin{equation}
\begin{aligned}
& A = \Im \left( \sum_{i=1}^3 \overline{u^i(x)} \nabla u^i(x) \right), \quad B = \Im \left( \sum_{i=1}^3 \overline{u^i(y)} \nabla u^i(y) \right), \\
& L_{\kappa_1 \kappa_2 \kappa_3 } = \sum_{i=1}^3 \kappa_i \vert \nabla u^i(x) \vert^2.
\end{aligned}
\end{equation}

$ \mathcal{A} $ remains itself unchanged because it will be treated as an error term below.

As for $ \mathcal{B} $, we make use of the decomposition identity $ 4 \psi + \phi = 5 \phi_{1} + 4 (\psi - \phi) + 5 (\phi - \phi_{1}) $ and deduce
\begin{equation}
\begin{aligned}
\mathcal{B} = & -\frac{10}{\omega_{5}R^5} \int_{\mathbb{R}^5} \int_{\mathbb{R}^5} \int_{\mathbb{R}^5} \Re \left( \overline{u^1(x)u^2(x)} u^3(x) \right) \chi^{3} \left( \frac{x - s}{R} \right) \chi^{2} \left( \frac{y - s}{R} \right) N_{\kappa_1 \kappa_2 \kappa_3} {\rm d}x {\rm d}y {\rm d}s \\
& -8 \int_{\mathbb{R}^5} \int_{\mathbb{R}^5} \Re \left( \overline{u^1(x)u^2(x)} u^3(x) \right) \left( \psi(x-y) - \phi(x-y) \right) N_{\kappa_1 \kappa_2\kappa_3} {\rm d}x {\rm d}y \\
& -10 \int_{\mathbb{R}^5} \int_{\mathbb{R}^5} \Re \left( \overline{u^1(x)u^2(x)} u^3(x) \right) \left( \phi(x-y) - \phi_{1}(x, y) \right) N_{\kappa_1\kappa_2\kappa_3} {\rm d}x {\rm d}y.
\end{aligned}
\end{equation}

We claim the quantity of $ \mathcal{C + E} $ is Galilean invariant, that is,
invariant under the the transformation
\begin{equation*}
\b{\rm u}(t, x) \mapsto \b{\rm u^{\xi}} = ( u^{1,\xi}(t, x), u^{2,\xi}(t, x),u^{3,\xi} )^T := ( e^{\frac{ix\cdot\xi}{\kappa_1}}u^1, e^{\frac{ix\cdot\xi}{\kappa_1}}u^2,e^{\frac{ix\cdot\xi}{\kappa_1}}u^3 )^T, \quad\quad
\end{equation*}
for any $\b{\rm u} = (u^1, u^2,u^3)^T$ and $ \xi = \xi(t, s, R) $.
In fact, we can compute
\[
\sum_{i=1}^3 \kappa_i \vert \nabla u^{i, \xi}(x) \vert^2 
= L_{\kappa_1\kappa_2\kappa_3} + \sum_{i=1}^3 2\kappa_i \xi \cdot \Im \left( \overline{u^i(x)} \nabla u^i(x) \right) + \sum_{i=1}^3 \frac{\vert \xi \vert^2}{\kappa_i} \vert u^i(x) \vert^2, 
\]
\[
\sum_{i=1}^3 \kappa_i \vert u^{i, \xi}(y) \vert^2 = \sum_{i=1}^3 \kappa_i \vert u^i(y) \vert^2, 
\]
and
\[
\Im \left( \sum_{i=1}^3 \overline{u^{i, \xi}} \nabla u^{i, \xi} \right)
= A + \left( \sum_{i=1}^3 \frac{\xi}{\kappa_i} \vert u^i(x) \vert^2 \right).
\]
Thus,
\begin{equation*}
\begin{aligned}
& \left( \sum_{i=1}^3 \kappa_i \vert \nabla u^{i, \xi}(x) \vert^2 \right) \left( \sum_{i=1}^3 \frac{\vert u^{i, \xi}(y) \vert^2}{\kappa_i} \right) - \Im \left( \sum_{i=1}^3 \overline{u^{i, \xi}} \nabla u^{i, \xi} \right)(x) \Im \left( \sum_{i=1}^3 \overline{u^{i, \xi}} \nabla u^{i, \xi} \right)(y)  \\
= & ~ L_{\kappa_1\kappa_2\kappa_3} \left( \sum_{i=1}^3 \frac{\vert u^i(y) \vert^2}{\kappa_i} \right) - A \cdot B + \xi \cdot A \left( \sum_{i=1}^3 \frac{\vert u^i(y) \vert^2}{\kappa_i} \right) - \left( \sum_{i=1}^3 \frac{\vert u^i(x) \vert^2}{\kappa_i} \right) \xi \cdot \Im \left( \sum_{i=1}^3 \overline{u^{i, \xi}(y)} \nabla u^{i, \xi}(y) \right),
\end{aligned}
\end{equation*}
and hence the claim follows by symmetry of $ \chi^{2} $ and a change of variables.

The choice of $ \xi = \xi(t, s, R) $ in \eqref{xi} implies that 
\begin{equation*}
\int_{\mathbb{R}^5} \Im \left( \sum_{i=1}^3 \overline{u^{i, \xi}} \nabla u^{i, \xi}(x) \right) \chi ^{2} \left(\frac{x-s}{R}\right) {\rm d}x = 0.
\end{equation*}
As a result,
\begin{equation}
\begin{aligned}
\mathcal{C + E} = \frac{4}{\omega_{5}R^5} \int_{\mathbb{R}^5 \times \mathbb{R}^5 \times \mathbb{R}^5} & ~ \left( \sum_{i=1}^3 \kappa_i \vert \nabla u^{i, \xi}(x) \vert^2 \right) \chi^{2} \left( \frac{x - s}{R} \right) \chi^{2} \left( \frac{y - s}{R} \right)\\
& ~ \cdot\left( \sum_{i=1}^3 \frac{\kappa_1 \kappa_2 \kappa_3}{\kappa_i} \vert u^i(y) \vert^2 \right) {\rm d}x {\rm d}y {\rm d}s. 
\label{C+E}
\end{aligned}
\end{equation}

Note that, by Cauchy-Schwartz inequality, 
\begin{equation*}
\Im \left( \sum_{i=1}^3 \overline {u^i}\not\!\nabla u^i \right) \le \sqrt{ \sum_{i=1}^3 \vert \not\!\nabla u^i \vert^2} \sqrt{\sum_{i=1}^3 \vert u^i \vert^2},
\label{D+F}
\end{equation*}
we can get $ \mathcal{D + F} \ge 0 $.

To conclude, we deduce
\begin{equation*}
\begin{aligned}
&\frac{\mathrm{d}}{\mathrm{d}t}M(t)\\
&\ge \int_{\mathbb{R}^5} \int_{\mathbb{R}^5} \Delta W_{\kappa_1 \kappa_2 \kappa_3} \left( 4 \psi(x-y) + \phi(x-y) \right) \left( \sum_{i=1}^3 \vert u^i(y) \vert^2 \right) {\rm d}x {\rm d}y \\
& -\frac{10}{\omega_{5}R^5} \int_{\mathbb{R}^5} \int_{\mathbb{R}^5} \int_{\mathbb{R}^5} \Re \left( \overline{u^1(x)u^2(x)} u^3(x) \right) \chi^{3} \left( \frac{x - s}{R} \right) \chi^{2} \left( \frac{y - s}{R} \right) N_{\kappa_1 \kappa_2 \kappa_3} {\rm d}x {\rm d}y {\rm d}s \\
& -8 \int_{\mathbb{R}^5} \int_{\mathbb{R}^5} \Re \left( \overline{u^1(x)u^2(x)} u^3(x) \right) \left( \psi(x-y) - \phi(x-y) \right) N_{\kappa_1 \kappa_2 \kappa_3 } {\rm d}x {\rm d}y \\
& -10 \int_{\mathbb{R}^5} \int_{\mathbb{R}^5} \Re \left( \overline{u^1(x)u^2(x)} u^3(x) \right) \left( \phi(x-y) - \phi_{1}(x, y) \right) N_{\kappa_1 \kappa_2 \kappa_3 } {\rm d}x {\rm d}y \\
& +\frac{4}{\omega_{5}R^5} \int_{\mathbb{R}^5} \int_{\mathbb{R}^5} \int_{\mathbb{R}^5} L^{\xi}_{\kappa_1\kappa_2\kappa_3} \chi^{2} \left( \frac{x - s}{R} \right) \chi^{2} \left( \frac{y - s}{R} \right) N_{\kappa_1\kappa_2\kappa_3} {\rm d}x {\rm d}y {\rm d}s \\
= &: \mathcal{A + G + H + I + J}.
\end{aligned}
\end{equation*}

Next, we will average this inequality over $ t \in I $ and logarithmically over $ R \in [R_0, R_0e^{J}] $.

Looking back at the definition of $ M(t) $, we find the upper bound $ \sup_{t \in \mathbb{R}} \vert M(t) \vert \lesssim RE_0^2 $. By the fundamental theorem of calculus, we have
\begin{equation}
\left\vert \frac{1}{T_0} \int_{I} \frac{1}{J} \int_{R_0}^{R_0e^{J}} \frac{\mathrm{d}}{\mathrm{d}t}M(t) \frac{{\rm d}R}{R} {\rm d}t \right\vert \lesssim_{\kappa_1 \kappa_2 \kappa_3} \frac{1}{T_0} \frac{R_0e^{J}}{J} E_0^2.
\end{equation}

We turn to $ \mathcal{A} $ integrating by parts.
\begin{equation*}
\begin{aligned}
\mathcal{A} \gtrsim_{\kappa_1\kappa_2\kappa_3} - \int_{\mathbb{R}^5} \int_{\mathbb{R}^5} & ~ \left( \sum_{i=1}^3 \vert u^i(x) \vert \vert \nabla u^i(x) \vert\right) \\
& ~ \left\vert 4 \nabla\psi(x-y) + \nabla\phi(x-y) \right\vert \left( \sum_{i=1}^3 \vert u^i(y) \vert^2 \right) {\rm d}x {\rm d}y.
\end{aligned}
\end{equation*}
The facts $ \vert \nabla \phi \vert \lesssim \frac{1}{R} $ and $ \vert \nabla \psi \vert = \vert \frac{x}{\vert x \vert^2} (\phi - \psi) \vert \lesssim \min \{ \frac{1}{R}, \frac{R}{\vert x \vert^2} \} $ tell us
\begin{equation}
\frac{1}{T_0} \int_{I} \frac{1}{J} \int_{R_0}^{R_0e^{J}} \mathcal{A} \frac{{\rm d}R}{R} {\rm d}t \gtrsim_{\kappa_1\kappa_2\kappa_3} - \frac{1}{JR_0} E_0^2.
\label{A}
\end{equation}

For $ \mathcal{G + J} $, we can establish a lower bound for these terms by Lemma \ref{coerb} after choosing $ \chi_{R}(x) = \chi \left( \frac{x-s}{R} \right) $, that is
\begin{equation}
\begin{aligned}
& \frac{1}{T_0} \int_{I} \frac{1}{J} \int_{R_0}^{R_0e^{J}} \mathcal{G + J} \frac{{\rm d}R}{R} {\rm d}t \\
\gtrsim & \frac{\delta}{JT_0} \int_{I} \int_{R_0}^{R_0e^{J}}\frac{1}{R^5} \int_{\mathbb{R}^5} \int_{\mathbb{R}^5} \int_{\mathbb{R}^5} L^{\xi}_{\kappa_1 \kappa_2 \kappa_3} \chi^{2} \left( \frac{x - s}{R} \right) \chi^{2} \left( \frac{y - s}{R} \right) N_{\kappa_1 \kappa_2 \kappa_3} {\rm d}x {\rm d}y {\rm d}s \frac{{\rm d}R}{R} {\rm d}t.
\end{aligned}
\label{G+J}
\end{equation}

As for $ \mathcal{H} $, by construction,
\begin{equation*}
\vert \psi(x) - \phi(x) \vert \lesssim \min \left\{ \frac{\vert x \vert}{R}, \frac{R}{\vert x \vert} \right\}.
\end{equation*}
We deduce
\begin{equation}
\frac{1}{T_0} \int_{I} \frac{1}{J} \int_{R_0}^{R_0e^{J}} \mathcal{H} \frac{{\rm d}R}{R} {\rm d}t \gtrsim_{\kappa_1 \kappa_2 \kappa_3 } - \frac{1}{J} E_0^2.
\label{H}
\end{equation}

Finally, similar to the estimates of $ \mathcal{H} $, we have
\begin{equation}
\frac{1}{T_0} \int_{I} \frac{1}{J} \int_{R_0}^{R_0e^{J}} \mathcal{I} \frac{{\rm d}R}{R} {\rm d}t \gtrsim_{\kappa_1 \kappa_2 \kappa_3} - \eps E_0^2
\label{I}
\end{equation}
since $ \vert \phi(x-y) - \phi_{1}(x, y) \vert \lesssim \eps $.

Collecting \eqref{A}, \eqref{G+J}, \eqref{H}, and \eqref{I}, we find
\begin{equation*}
\begin{aligned}
&\frac{\delta}{JT_0} \int_{I} \int_{R_0}^{R_0e^{J}}\frac{1}{R^5} \int_{\mathbb{R}^5} \int_{\mathbb{R}^5} \int_{\mathbb{R}^5} L^{\xi}_{\kappa_1 \kappa_2 \kappa_3} \chi^{2} \left( \frac{x - s}{R} \right) \chi^{2} \left( \frac{y - s}{R} \right) N_{\kappa_1 \kappa_2 \kappa_3} {\rm d}x {\rm d}y {\rm d}s \frac{{\rm d}R}{R} {\rm d}t \\
&\lesssim_{\kappa_1\kappa_2\kappa_3}  \left( \frac{R_0e^{J}}{JT_0} + \frac{1}{JR_0} + \frac{1}{J} + \eps \right) E_0^2,
\end{aligned}
\end{equation*}
which completes the proof of Proposition \ref{decay}.
\end{proof}
\section{Proof of the main result}
In this section, we combine the results in Section 4 and Section 5 to complete the proof of Theorem \ref{main}.

\textbf{The proof of Theorem \ref{main}:}

First of all, using the scaling
\[
\b{\rm u}_{\lambda}(t, x) = \lambda^2 \b{\rm u} \left( \lambda^2 t, \lambda x \right),
\]
we can always fix $ \lambda > 0 $ such that \eqref{E0} holds.
In order to establish \eqref{b1}, we change \eqref{Me} into
\[
\frac{\delta}{JT_0} \int_{I} \int_{R_0}^{R_0e^{J}} \frac{1}{R^5} \int_{\mathbb{R}^5} \int_{\mathbb{R}^5} \int_{\mathbb{R}^5} \mathcal{L}^{\xi}_{\kappa_1 \kappa_2 \kappa_3} \mathcal{N}_{\kappa_1 \kappa_2 \kappa_3} {\rm d}x {\rm d}y {\rm d}s \frac{{\rm d}R}{R} {\rm d}t \lesssim_{\kappa_1  \kappa_2 \kappa_3 } \left( \frac{R_0 e^{J}}{J T_0} + \eps \right) E_0^2,
\]
where
\[
\mathcal{L}^{\xi}_{\kappa_1 \kappa_2 \kappa_3} = \sum_{i=1}^3 \kappa_i \left\vert \nabla \left( \chi \left( \frac{x-s}{R} \right) u^{i, \xi}(x) \right) \right\vert^2,
\]
and
\[
\mathcal{N}_{\kappa_1 \kappa_2 \kappa_3} = \kappa_1 \kappa_2 \kappa_3 \sum_{i=1}^3 \frac{ \left\vert \chi^{2} \left( \frac{y - s}{R} \right) u^i(y) \right\vert^2}{\kappa_i}.
\]
If we choose $ J = \eps^{-1} R_0, T_0 = e^{J} $, then
\[
\frac{\delta}{JT_0} \int_{I} \int_{R_0}^{R_0e^{J}} \frac{1}{R^5} \int_{\mathbb{R}^5} \int_{\mathbb{R}^5} \int_{\mathbb{R}^5} \mathcal{L}^{\xi}_{\kappa_1\kappa_2 \kappa_3} \mathcal{N}_{\kappa_1\kappa_2 \kappa_3} {\rm d}x {\rm d}y {\rm d}s \frac{{\rm d}R}{R} {\rm d}t \lesssim_{\kappa_1 \kappa_2 \kappa_3 , E_0} \eps.
\]

Considering the support of $ \chi $ in \eqref{Gamma} and using the integral mean value theorem, we can omit writing some constants and get
\[
\frac{1}{T_0} \int_{I} \left\Vert \left( \chi u^1, \chi u^2, \chi u^3 \right) \right\Vert^2_{{\rm \dot{H}}_{x}^1(\mathbb{R}^5)} \left\Vert \left( \chi u^1, \chi u^2, \chi u^3 \right) \right\Vert^2_{{\rm L}_{y}^2(\mathbb{R}^5)} {\rm d}t
\lesssim_{\kappa_1 \kappa_2 \kappa_3, E_0} \eps.
\]
The interval $ I $ is divided into many sub-intervals of the same length $ l = \eps^{-\frac{1}{4}} $.
By the pigeonhole principle, there exists $ I_{k} $ such that
\begin{equation*}
\int_{I_{k}} \left\Vert \left( \chi u^1, \chi u^2, \chi u^3 \right) \right\Vert^2_{{\rm \dot{H}}_{x}^1(\mathbb{R}^5)} \left\Vert \left( \chi u^1, \chi u^2, \chi u^3 \right) \right\Vert^2_{{\rm L}_{y}^2(\mathbb{R}^5)} {\rm d}t
\lesssim_{E_0} \eps^{\frac{3}{4}},
\end{equation*}
which means
\begin{equation*}
\left\Vert \left( \chi u^1, \chi u^2, \chi u^3 \right) \right\Vert_{{\rm L}_t^4(I_k, {\rm L}_x^{\frac{5}{2}}(\mathbb{R}^5))}^4=\int_{I_{k}} \left\Vert \left( \chi u^1, \chi u^2, \chi u^3 \right) \right\Vert^4_{{\rm L}^{\frac{5}{2}}(\mathbb{R}^5)} {\rm d}t \lesssim_{E_0} \eps^{\frac{3}{4}}.
\end{equation*}
Thanks to H\"older inequality, we have
\begin{equation*}
\begin{aligned}
\left\Vert \left( \chi u^1, \chi u^2, \chi u^3 \right) \right\Vert_{{\rm L}_{t}^2(I_{k}, {\rm L}_{x}^{\frac{5}{2}}(\mathbb{R}^5))}
\le & ~ \left\Vert \left( \chi u^1, \chi u^2, \chi u^3 \right) \right\Vert_{{\rm L}_t^4(I_k, {\rm L}_x^{\frac{5}{2}}(\mathbb{R}^5))} \left\Vert \b{\rm 1} \right\Vert_{{\rm L}_t^4(I_k, {\rm L}_x^\infty(\mathbb{R}^5))} \\
\lesssim&_{E_0} ~ \eps^{\frac3{16}} \eps^{-\frac1{16}}
= \eps^{\frac18}.
\end{aligned}
\end{equation*}
At last, by interpolation,
\begin{equation*}
\begin{aligned}
& ~ \left\Vert \left( \chi u^1, \chi u^2, \chi u^3 \right) \right\Vert_{{\rm L}_{t}^6(I_{k}, {\rm L}_{x}^{3}(\mathbb{R}^5))} \\
\le & ~ \left\Vert \left( \chi u^1, \chi u^2, \chi u^3 \right) \right\Vert^{\frac{1}{3}}_{{\rm L}_{t}^2(I_{k}, {\rm L}_{x}^{\frac{5}{2}}(\mathbb{R}^5))}
\left\Vert \left( \chi u^1, \chi u^2, \chi u^3\right) \right\Vert^{\frac{2}{3}}_{{\rm L}_{t}^{\infty}(I_{k}, {\rm H}_{x}^1(\mathbb{R}^5))}
\lesssim_{E_0} \eps^{\frac{1}{24}}, 
\end{aligned}
\end{equation*}
which implies \eqref{b1} holds.

So the scattering criterion, Proposition \ref{scat}, tells us that $ \b{\rm u} $ scatters, which completes the proof of Theorem \ref{main}.

\end{document}